\documentclass[a4paper,12pt]{amsart}
\usepackage{a4wide}
\usepackage{amsmath,amsfonts,amssymb,amsthm}
\usepackage{graphics}
\usepackage{epsfig}
\usepackage{esint}
\usepackage{mathrsfs}

\usepackage[colorlinks,citecolor=blue]{hyperref}

\newcommand{\ignore}[1]{}

\numberwithin{equation}{section}

\newtheorem{theorem}{Theorem}

\newtheorem{lemma}{Lemma}

\newtheorem{assumption}{Assumption}

\newcommand{\R}{\mathbb{R}}

\newcommand{\dd}{\mathrm d}
\def\les{\lesssim}
\def\ges{\gtrsim}
\def\lt{\left}
\def\rt{\right}
\def \dd{\mathrm d}
\def \e{\varepsilon}
\renewcommand{\tilde}{\widetilde}
\usepackage{color}
\definecolor{darkred}{rgb}{0.9,0.1,0.1}
\definecolor{darkblue}{rgb}{0,0,0.7}
\definecolor{darkgreen}{rgb}{0,0.5,0}

\begin{document}
\title[Partial regularity for OT maps with $p-$cost]{Partial regularity for optimal transport with \texorpdfstring{$p$}{}-cost away from fixed points}

\author{Michael Goldman}
\address{CMAP, Polytechnique, Route de Saclay, 91128 Palaiseau Cedex, France}
\email{michael.goldman@cnrs.fr}

\author{Lukas Koch}
\address{Max Planck Institute for Mathematics in the Sciences, Inselstrasse 22, 04229 Leipzig, Germany}
\email{lkoch@mis.mpg.de}

\begin{abstract}
We consider maps $T$ solving the optimal transport problem with a cost $c(x-y)$ modeled on the $p$-cost. For H\"older continuous marginals, we prove a $C^{1,\alpha}$-partial regularity result for $T$ in the set $\{\lvert T(x)-x\rvert>0\}$.
\end{abstract}

\maketitle

\section{Introduction}
In this paper we are concerned with the optimal transportation problem
\begin{align}\label{eq:problem}
\min_{\pi\in \Pi(\rho_0,\rho_1)} \int c(x-y)\mathrm{d}\pi,
\end{align}
where $\Pi(\rho_0,\rho_1)$ denotes the set of plans in $\R^d\times \R^d$ with marginals $\rho_0$ and $\rho_1$. Under mild assumptions, see e.g. \cite[Theorem 2.12]{Villani}, it is known that the minimiser is unique and  of Monge-form, that is $\pi = (x,T(x))_{\# \rho_0}$ for some map $T\colon \R^d\to \R^d$. We are concerned with obtaining a partial regularity result for $T$ in the setting where $c$ is modeled on the cost $\lvert x-y\rvert^p$ for some $p>1$. 

In order to state our main theorem, let us make our assumptions on the cost function precise. 

\begin{assumption}\label{ass}
Let $p>1$ and $\alpha\in(0,1]$. We consider a convex cost function $c\in C^{2,\alpha}(\R^d\setminus \{0\})\cap C^1(\R^d)$ satisfying $c(0)=0$, $\nabla c(0)=0$ and for some $\Lambda>0$,
\begin{align}
\Lambda^{-1}(\lvert x\rvert+\lvert y\rvert)^{p-2}\lvert x-y\rvert^2\leq& \langle \nabla c(x)-\nabla c(y),x-y\rangle,\quad& \forall x,y\in \R^d\,\label{ass1}\\
\lvert \nabla c(x)-\nabla c(y)\rvert \leq& \Lambda (\lvert x\rvert+\lvert y\rvert)^{p-2}\lvert x-y\rvert,\quad &\forall x,y\in \R^d\label{ass2}.
\end{align}
\end{assumption}
Let 
\[
 X=\{x\colon \lvert T(x)-x\rvert>0\} \qquad \textrm{and } \qquad  Y=\{y\colon \lvert y-T^{-1}(y)\rvert>0\}.
\]
Our main result then reads as follows:
\begin{theorem}\label{thm:main}
Let $p>1$. Suppose $c$ satisfies Assumption \eqref{ass} and let $\rho_0,\rho_1\in C^\alpha(\R^d)$. Suppose $\pi=(x,T(x))_{\#\rho_0}$ solves \eqref{eq:problem}. Then, there exist open sets $X^\prime$ and $Y^\prime$ with $|X\backslash X'|=|Y\backslash Y'|=0$ and  that $T$ is a $C^{1,\alpha}$-diffeomorphism between the sets $X'$ and $Y'$.
\end{theorem}

Partial regularity theorems for \eqref{eq:problem} were obtained by \cite{Figalli2010} in the case $\lvert x-y\rvert^2$ and for general non-singular cost $c$ in \cite{DePhilippis2015} using \textsc{Caffarelli's} viscosity approach to solutions of the Monge-Amp\`{e}re equation, which is the Euler-Lagrange equation for \eqref{eq:problem}. In \cite{Goldman2018,Goldmanrendiconti} the result of \cite{Figalli2010} was obtained using a variational approach based on the geometric linearisation of the Monge-Amp\`{e}re equation, see also \cite{GHO}. This was later generalised to reprove the results of \cite{DePhilippis2015} in \cite{Prodhomme2021}. Regarding singular cost functions of $p-$type, to the best of our knowledge, the only known result is the global regularity for densities defined on convex domains which are either sufficiently far apart or  under the condition that  $p\ge2$ with $|p-2|\ll1$, see \cite{Cafp}. In this note we thus obtain partial regularity when the domains are neither convex nor disjoint. 

The main (and very natural) idea behind our proof is that away from fixed points, that is for $x\neq 0$, $c$ locally looks like a quadratic cost function. In order to localize and apply  the existing partial regularity theory for non-singular cost-functions from \cite{DePhilippis2015,Prodhomme2021} we need a localized $L^\infty$ bound on the displacement. When $p\ge2$, since the Kantorovich potentials are semi-convex, this follows from Alexandrov theorem (see \cite{DePhilippis2015,Prodhomme2021}) while the situation is more delicate for $p<2$. Our second main result is the following $L^\infty$-bound.
\begin{theorem}\label{thm:linfty}
Let $p>1$ and  $\pi$ be a minimiser for the $c$-optimal transport problem between  $\rho_0$ and $\rho_1$. Given $A\in \R^{d\times d}$, $b\in \R^d$ and $m>0$, denote
\begin{align*}
E_m= \int_{B_2} \lvert T(x)-x-b-Ax\rvert^{m}\dd\rho_0,\quad D = \inf_{B_r\Subset B_2} \frac{\rho_0(B_r)}{r^d}.
\end{align*}
Assuming $D>0$, it holds that for $x\in B_{1/2}$,
\begin{align*}
\lvert T(x)-x-b-Ax\rvert \lesssim \begin{cases} E_m^\frac 1 {m} \qquad &\text{ if } E_m\geq 1\\
E_m^\frac 1 {m+d} &\text{ if } E_m\leq 1.
\end{cases}
\end{align*}
The implicit constant depends on $D$, $p$, $m$, $d$ and $\lvert A\rvert$, but not on $b$.
\end{theorem}

In order to prove Theorem \ref{thm:main}, we will apply Theorem \ref{thm:linfty} with $m = p$, in which case $E_p$ is a standard excess quantity.\\
Theorem \ref{thm:linfty} is of interest in its own right and might actually be seen as the main contribution of the note. Global versions of this bound with $m=p$, $b=0$ and $A=0$ first appeared  in \cite{Bouchitte2007}. As far as local bounds are concerned, in the case $p=2$, where by affine invariance of the cost it is possible to reduce to the case $b=0$ and $A=0$, such bounds were obtained  in \cite{Goldman2018,Prodhomme2021,GHO}. These bounds  played a major role in the variational approach to regularity theory  for optimal transport maps. Still when $b=0$ and $A=0$ but for $p\neq 2$, Theorem \ref{thm:linfty} recovers earlier results of \cite{Koch2024,Gutierrez2021}. In the case $p\ge 2$ (and for homogeneous costs), Theorem \ref{thm:linfty} has been proven in \cite{Gut2024} (see also the first Arxiv version of \cite{Gutierrez2022}) with a much more complicated proof.\\


 As alluded to, when $p\neq2$, we lose the  affine invariance of the cost and Theorem \ref{thm:linfty} cannot be reduced to the case $b=0$ and $A=0$ anymore. Moreover, in this case, the approach from \cite{Goldman2018,Prodhomme2021} with the geometric interpretation from \cite{Koch2024} seems difficult to implement. Our argument is much closer in spirit  to the original proof in \cite{Bouchitte2007}. As in \cite{Bouchitte2007}, the idea is to use the $c$-monotonicity optimality condition for many points which lie on the ray $T(x)-x-b$.\\

 We anticipate that Theorem \ref{thm:linfty} might have applications for the  derivation of a full partial regularity result (also including fixed points). This motivated us to include the case $A\neq 0$ in our statement. Let us however point out that an $\e-$regularity theorem around fixed points seems to be a very challenging problem for multiple reasons. First, even though an 'harmonic-approximation' result has been obtained in this context in \cite{Koch2024}, the approach of \cite{Goldman2018} works at the level of $C^{1,\alpha}$ regularity for the map $T$ which requires $C^{2,\alpha}$ regularity for the solutions of the geometrically linearized equation (here the $p'-$Laplace equation). However such regularity is known to fail for $p> 2$ and is only known qualitatively for $d=2$ and $p<2$, see \cite{Iwaniec1989}. Second, even if we ignore this issue, due to the lack of affine invariance of the cost, after one-step of the Campanato iteration we lose the structure of having a cost of the form $c(x-y)$ which is essential for the 'harmonic-approximation' results of \cite{Goldman2018,GHO,Koch2024}. Last, in the case $p<2$, where one could hope to get the needed $C^{2,\alpha}$ regularity for the $p'-$Laplace equation, the Kantorovich potentials are not semi-convex anymore and we cannot appeal to Alexandrov theorem to prove that around most fixed point the excess is small. See however \cite{DuzMin} where a related partial regularity results has been obtained for the geometrically linearized problem. We thus leave this problem for future research.\\


The outline of our paper is as follows. In Section \ref{sec:prelim} we collect some preliminary comments from convex analysis and regarding certain quantities related to $p$-cost. In Section \ref{sec:linfty} we prove Theorem \ref{thm:linfty}, while we prove Theorem \ref{thm:main} in Section \ref{sec:mainTheorem}

\section*{Notation}
The symbols $\simeq$, $\ges$, $\les$ indicate estimates that hold up to a global constant $C$,
which only depends on the dimension $d$, the exponent $p$ and the H\"older exponent $\alpha$ (if applicable). 
For instance, $f\les g$ means that there exists such a constant with $f\le Cg$,
$f\simeq g$ means $f\les g$ and $g\les f$. An assumption of the form $f\ll1$ means that there exists $\e>0$, only
depending on $p$, $d$ and $\alpha$, such that if $f\le\e$, 
then the conclusion holds.   
\section{Preliminaries}\label{sec:prelim}
We collect some well-known facts from convex analysis and regarding certain quantities naturally related to the type of cost functions we consider.

Let $c\in C^1(\R^d)$ be a convex function. Then for any $x,y\in \R^d$,
\begin{align}\label{eq:convex1}
c(x)-c(y)\leq \langle \nabla c(x),x-y\rangle.
\end{align}

Furthermore if $c$ is strictly convex, $C^1$ and superlinear, then $\nabla c$ is a homeomorphism of $\R^d$.

\eqref{ass1} is equivalent to the assumption that for some $\Lambda>0$ and all $x,y\in \R^d$,
\begin{align}\label{ass1a}
c(x)\geq c(y)+\langle \nabla c(y),x-y\rangle+\Lambda^{-1} (\lvert x\rvert+\lvert y\rvert)^{p-2}\lvert x-y\rvert^2.
\end{align}
A further equivalent formulation is that for some $\Lambda>0$ and all $x,y\in \R^d$, $\lambda\in [0,1]$,
\begin{align}\label{ass1b}
c(\lambda x+(1-\lambda) y)\leq \lambda c(x)+(1-\lambda) c(y)-\Lambda^{-1}\lambda(1-\lambda)(\lvert x\rvert+\lvert y\rvert)^{p-2}\lvert x-y\rvert^2.
\end{align}

Note that \eqref{ass2}  with $y=0$ combined with \eqref{ass1} for $y=0$ imply $\lvert \nabla c(x)\rvert \sim \lvert x\rvert^{p-1}$. 
Then applying \eqref{ass1a} with $x=0$ or $y=0$ in combination with this fact, we obtain that
\begin{align}\label{growth}
c(x)\sim \lvert x\rvert^p.
\end{align}

Notice also that \eqref{ass1} together with \eqref{ass2} and $c\in C^2(\R^d\backslash\{0\})$ implies 
\begin{equation}\label{hypHessc}
 |\nabla^2 c(z)|\simeq |z|^{p-2}.
\end{equation}

\begin{lemma}\label{lem:monotone}
Let $p>1$ and $\beta\geq 1$. Then for any $\mu\geq 0$, the function
\begin{align*}
t\to (\mu+t)^{p-2}t^\beta
\end{align*}
is monotonic increasing in $t$. In fact, for every $C_0>0$ there exists $C_1>0$ such that  for every  $s,t,\mu \geq 0$,
\begin{align}\label{eq:monoton}
(\mu+t)^{p-2}t^\beta\le C_0(\mu+s)^{p-2}s^\beta\quad\Rightarrow \quad t\le C_1 s.
\end{align}
Finally, a Young-type inequality holds: For any $\e>0$ there is $C_\e>0$ such that for any $\mu\geq 0$, $x,y\in \R^d$,
\begin{align}\label{eq:Young}
(\mu+\lvert x\rvert)^{p-2}\lvert x\rvert \lvert y \rvert \leq \e (\mu+\lvert x\rvert)^{p-2}\lvert x\rvert^2+C_\e (\mu+\lvert y\rvert)^{p-2}\lvert y\rvert^2.
\end{align}
\end{lemma}
\begin{proof}
Monotonicity of $t\mapsto (\mu+t)^{p-2}t^\alpha$ follows directly from checking that the derivative of the function is non-negative. Let $\gamma=\min(\beta,p-2+\beta)>0$ and for  $C_0>0$ set  $C_1=C_0^{1/\beta}$. If $s,t\geq 0$ are such that
\begin{align*}
(\mu+t)^{p-2}t^\beta\leq C_0 (\mu+s)^{p-2}s^\beta,
\end{align*}
assume for the sake of contradiction that  $t> C_1 s$. Using the monotonicity we find 
\begin{align*}
C_0 (\mu+s)^{p-2}s^\beta&\ge (\mu+t)^{p-2}t^\beta > (\mu+C_1s)^{p-2} (C_1s)^\beta \\
&\geq C_1^{\gamma}(\mu+s)^{p-2}s^\beta=C_0 (\mu+s)^{p-2}s^\beta.
\end{align*}
This  gives the desired contradiction. \\

A proof of \eqref{eq:Young} can be found in \cite[Lemma 2.3]{Mingione2002}.
\end{proof}

\section{\texorpdfstring{$L^\infty$}{}-bounds}\label{sec:linfty}
In this section we prove Theorem \ref{thm:linfty}.

\begin{proof}[Proof of Theorem \ref{thm:linfty}]

\textbf{Step 1: Cyclical monotonicity} Let $\tilde{x}\in B_{1/2}$ and write $\tilde{y}=T(\tilde{x})$. Making the change of variables 
\[
 x=\tilde{x}+x', \quad y=\tilde{x}+y', \quad \rho_0'(x')=\rho_0(x) \quad \textrm{and } \quad \rho_1'(y')=\rho_1(y)
\]
we have by translation invariance of the cost that  $T'(x')=T(x)-\tilde{x}$ is the optimal transport map between $\rho'_0$ and $\rho'_1$. Setting $\tilde{b}=b+A\tilde{x}$ and 
\[
 \tilde{E}=\int_{B_1} |T'-x'-\tilde{b} -Ax'|^m d\rho_0',
\]
it is enough to prove that 
\begin{align}\label{eq:Linfty2}
\lvert T'(0)-\tilde b\rvert \lesssim \begin{cases} \tilde{E}^\frac 1 {m} \qquad &\text{ if } \tilde{E}\geq 1\\
\tilde{E}^\frac 1 {m+d} &\text{ if } \tilde{E}\leq 1.
\end{cases}
\end{align}
To lighten notation we write $T$ for $T'$ and $E$ for $\tilde{E}$. Moreover we denote $y=\tilde{y}-\tilde{x}$.
Given $(x^\prime,y^\prime)=(x^\prime,T(x^\prime))\in B_1\times \R^d$, by cyclical monotonicity we have
\begin{align}\label{eq:mono2}
c(y)+c(y^\prime-x^\prime)\leq c(y^\prime)+c(y-x^\prime).
\end{align}
Let us point out that this is  the only form in which we use the minimality of $T$.

\textbf{Step 2: Preliminary estimates}
Note that for 
\begin{align}\label{eq:erestr1}
\e< \frac 1 4\lvert A\rvert^{-1},
\end{align}
we may define $v_\e$ uniquely as a solution of the equation
\begin{align}\label{eq:defv}
v_\e = \e \left(y-(A v_\e+\tilde b)\right) \quad \Leftrightarrow \quad v_\e = (\textup{Id}+\e A)^{-1}\e(y-\tilde b)
\end{align}
Further note that, if \eqref{eq:erestr1} holds and we additionally ensure
\begin{align}\label{eq:erestr2}
\e \lvert  y-\tilde b\rvert\leq \frac 3 8,
\end{align}
then 
\begin{align*}
\lvert v_\e\rvert \leq \|\textup{Id}+\e A\|^{-1} |\e( y-\tilde b)|\leq \frac 1 2.
\end{align*}

We will precise our choice of $\e\in(0,1/2)$ at a later stage, but for now assume that \eqref{eq:erestr1} and \eqref{eq:erestr2} are satisfied. Then for $r\in (0,\lvert v_\e\rvert/2)$ to be fixed at a later stage, we may find $(x^\prime,y^\prime)=(x^\prime,T(x^\prime))\in B_1\times \R^d$ such that
\begin{align}\label{eq:xprime}
\lvert x^\prime- v_\e\rvert \leq r,\qquad \lvert y^\prime-x^\prime- b^\prime\rvert^m\leq r^{-d}\frac{E}{D}.
\end{align}
where 
\begin{equation}\label{defb'}
b^\prime = \tilde{b}+ A x^\prime.
\end{equation}
Indeed, suppose that  for all $(x^\prime,T(x^\prime))\in B_r(v_\e)\times \R^d\Subset B_2\times \R^d$ and some $M>D^{-1}$,
\begin{align*}
\lvert y^\prime-x^\prime-b^\prime\rvert^m\geq \frac{M E}{r^d}.
\end{align*}
Then
\begin{align*}
E\geq \frac{M E}{r^d}\int_{B_r(v_\e)}\mathrm d\rho_0\geq M E D,
\end{align*}
giving a contradiction.

Note that since $r\in (0, |v_\e|/2)$, definition \eqref{eq:defv} of $v_\e$ and \eqref{eq:erestr1}, and definition \eqref{eq:xprime} of $x^\prime$,
\begin{equation}\label{simeqvepsyb}
\lvert x^\prime\rvert+ |v_\e|\simeq|v_\e|\simeq \e|y-\tilde b|.
\end{equation}
Hence it also holds that
\begin{equation}\label{simeqybDelta}
 |y- \tilde b -Av_\e|= \e^{-1}\lvert v_\e\rvert\simeq |y-\tilde b|.
\end{equation}

Finally, using \eqref{simeqvepsyb} and the triangle inequality, we find
\begin{align}\label{error3}
 \lvert \tilde b + A v_\e+ x^\prime\rvert+\lvert  y- x^\prime\rvert\simeq |\tilde b|+| y|.
\end{align}

\textbf{Step 3: Applying cyclical monotonicity}

We now apply \eqref{eq:mono2} with the choice of $( x^\prime, y^\prime)$ given in \eqref{eq:xprime} and re-write \eqref{eq:mono2} as
\begin{align}\label{eq:mono3}
&c(y)+c(\tilde b+A v_\e)-c(\e y + (1-\e)(\tilde b + A v_\e))-c(\e (\tilde b + A v_\e)+(1-\e) y)\nonumber\\
\leq& \lt[\left(c(\tilde b+A v_\e+x^\prime)-c(\tilde b +A v_\e+ v_\e)\right)+\left(c(y-x^\prime)-c(y-v_\e)\right)\rt]\nonumber\\
&\quad+\lt[\left(c(\tilde b + A v_\e)-c(\tilde b + A v_\e+x^\prime)\right)+\left(c(y^\prime)-c(y^\prime-x^\prime)\right)\rt]\nonumber\\
=& I + II.
\end{align}

Using convexity in the form of \eqref{eq:convex1} followed by \eqref{ass2}, we find
\begin{align*}
I \lesssim& \langle \nabla c(\tilde b + A v_\e+x^\prime)-\nabla c(y-x^\prime),x^\prime-v_\e\rangle\nonumber\\
\lesssim& (\lvert \tilde b + A v_\e+x^\prime\rvert+\lvert y-x^\prime\rvert)^{p-2}\lvert \tilde b + A v_\e-y+2x^\prime\rvert \lvert  x^\prime-v_\e\rvert\nonumber\\
\stackrel{\eqref{eq:xprime}}{\leq}& (\lvert \tilde b + A v_\e+ x^\prime\rvert+\lvert  y- x^\prime\rvert)^{p-2}\lvert \tilde b + A v_\e-y+2 x^\prime\rvert r.
\end{align*}

Using \eqref{simeqvepsyb} and \eqref{error3}, we may further simplify the estimate above to
\begin{align}\label{eq:upper2}
I\lesssim& (\lvert \tilde b\rvert+\lvert y\rvert)^{p-2}\lvert \tilde b- y\rvert r.
\end{align}

Similarly,
\begin{align}\label{eq:upper1}
II\lesssim& \langle \nabla c(y^\prime)-\nabla c(\tilde b + A v_\e), x^\prime\rangle\nonumber\\
\lesssim& (\lvert y^\prime\rvert+\lvert \tilde b+Av_\e\rvert)^{p-2}\lvert y^\prime-\tilde b-Av_\e\rvert \lvert x^\prime\rvert\nonumber\\
\lesssim& (\lvert y^\prime-\tilde b-Av_\e\rvert+\lvert\tilde b+Av_\e\rvert)^{p-2}\lvert y^\prime-\tilde b-Av_\e\rvert \lvert x^\prime\rvert= II'.
\end{align}
On the one hand, if $\lvert y^\prime-\tilde b-Av_\e\rvert \lesssim \e\lvert y-\tilde b-Av_\e\rvert$, then we deduce using  the monotonicity from Lemma \ref{lem:monotone}, as well as \eqref{simeqvepsyb} and \eqref{simeqybDelta},
\begin{align*}
II'\lesssim& \e^{\min(2,p)} (\lvert y-\tilde b-Av_\e\rvert+\lvert \tilde b\rvert)^{p-2}\lvert y-\tilde b-Av_\e\rvert \lvert y-\tilde b\rvert\\
\lesssim& \e^{\min(2,p)} (\lvert y-\tilde b-Av_\e\rvert+\lvert \tilde b\rvert)^{p-2}\lvert y-\tilde b\rvert^2\\
\stackrel{\eqref{eq:xprime}}{\lesssim}& \e^{\min(2,p)} (\lvert y-\tilde b\rvert+\lvert \tilde b\rvert)^{p-2}\lvert y-\tilde b\rvert^2.
\end{align*}
On the other hand, if $\lvert y^\prime-\tilde b-Av_\e\rvert \gg \e\lvert y-\tilde b-Av_\e\rvert$, using \eqref{simeqvepsyb} and \eqref{simeqybDelta},
\begin{align}\label{eq:estim1}
\lvert x^\prime\rvert+\lvert A v_\e\rvert \lesssim \e\lvert y-\tilde b\rvert\lesssim \e\lvert y-\tilde b-Av_\e\rvert\ll \lvert  y^\prime-\tilde b-Av_\e\rvert.
\end{align}
Recalling the definition \eqref{defb'} of $b'$  we have by triangle inequality,
\begin{align*}
\lvert y^\prime-\tilde b-A v_\e\rvert\leq& \lvert y^\prime-b^\prime- x^\prime\rvert+\lvert A v_\e\rvert + \lvert x^\prime\rvert+\lvert A x^\prime\rvert.
\end{align*}
Using \eqref{eq:estim1} to absorb the last three terms on the right-hand side, we have in this case,
\begin{align}\label{otherhand}
\lvert y^\prime-\tilde b-A v_\e\rvert\lesssim \lvert y^\prime-b^\prime- x^\prime\rvert.
\end{align}
Using the monotonicity of Lemma \ref{lem:monotone}, this allows to estimate in this case 
\begin{align*}
II' \lesssim& (\lvert y^\prime-b^\prime- x^\prime\rvert+\lvert \tilde b+Av_\e\rvert)^{p-2}\lvert y^\prime- b^\prime- x^\prime\rvert \lvert x^\prime\rvert.
\end{align*}
Using now  \eqref{eq:estim1} and \eqref{otherhand} we find 
\[
 |y'-b'-x'|+\lvert \tilde b+Av_\e\rvert\simeq |y'-b'-x'|+|\tilde{b}|
\]
so that we can post-process it to 
\[
 II' \lesssim (\lvert y^\prime-b^\prime- x^\prime\rvert+\lvert \tilde b\rvert)^{p-2}\lvert y^\prime- b^\prime- x^\prime\rvert \lvert x^\prime\rvert.
\]
Combining both cases we conclude that 
\begin{equation}\label{II}
 II\les \e^{\min(2,p)} (\lvert y-\tilde b\rvert+\lvert \tilde b\rvert)^{p-2}\lvert y-\tilde b\rvert^2+(\lvert y^\prime-b^\prime- x^\prime\rvert+\lvert \tilde b\rvert)^{p-2}\lvert y^\prime- b^\prime- x^\prime\rvert \lvert x^\prime\rvert.
\end{equation}
Putting \eqref{eq:upper2} and \eqref{II} together yields
\begin{multline}\label{I+II}
 I+II\les \e^{\min(2,p)} (\lvert y-\tilde b\rvert+\lvert \tilde b\rvert)^{p-2}\lvert y-\tilde b\rvert^2\\
 +(\lvert \tilde b\rvert+\lvert y\rvert)^{p-2}\lvert \tilde b- y\rvert r+(\lvert y^\prime-b^\prime- x^\prime\rvert+\lvert \tilde b\rvert)^{p-2}\lvert y^\prime- b^\prime- x^\prime\rvert \lvert x^\prime\rvert.
\end{multline}

We now estimate  $I+II$ from below. As $\e\in(0,1)$, \eqref{ass1b} applied to the third and fourth term on the left-hand side in \eqref{eq:mono3} allows us to obtain the following lower bound (up to a constant depending on $p$ only):
\begin{align*}
\e(1-\e) (\lvert \tilde b+A v_\e\rvert+\lvert y\rvert)^{p-2}\lvert y-\tilde b-Av_\e\rvert^2\les I+II.
\end{align*}
Choosing $\e \ll 1$ and employing \eqref{simeqybDelta}, we may replace this by
\begin{align}\label{lowerI+II}
\e(\lvert \tilde b\rvert+\lvert y-\tilde b\rvert)^{p-2}\lvert y-\tilde b\rvert^2\les I+II.
\end{align}

Collecting estimates \eqref{I+II} and \eqref{lowerI+II} together and choosing $\e \ll 1$ to absorb the first right-hand side term in \eqref{I+II}, we have shown
\begin{align*}
\e (\lvert \tilde b\rvert+\lvert y-\tilde b\rvert)^{p-2} \lvert y-\tilde b\rvert^2\lesssim& (\lvert \tilde b\rvert+\lvert y\rvert)^{p-2}\lvert \tilde b- y\rvert r\nonumber\\
&+(\lvert y^\prime-b^\prime-x^\prime\rvert+\lvert \tilde b\rvert)^{p-2}\lvert y^\prime-b^\prime-x^\prime\rvert \lvert x^\prime\rvert
\end{align*}
where the implicit constant depends on $p$, $d$ and $\lvert A\rvert$. By \eqref{eq:xprime} and  the monotonicity from Lemma \ref{lem:monotone}, this implies
\begin{align}\label{eq:main2}
\e (\lvert \tilde b\rvert+\lvert y-\tilde b\rvert)^{p-2}\lvert y-\tilde b\rvert^2\lesssim& (\lvert \tilde b\rvert+\lvert y-\tilde b\rvert)^{p-2}\lvert\tilde b-y\rvert r\\
&\quad+ ((r^{-d}E)^\frac 1 m+\lvert \tilde b\rvert)^{p-2}(r^{-d} E)^\frac 1 m \lvert x^\prime\rvert \nonumber.
\end{align}

\textbf{Step 4: Proof of \eqref{eq:Linfty2} for $E\geq 1$:} As $E\geq 1$ it suffices to show 
\begin{align}\label{eq:conclude2}
\lvert y-\tilde b\rvert \lesssim 1+E^\frac 1 m.
\end{align}
Thus, we may assume without loss of generality that $\lvert y-\tilde b\rvert \geq 1$.
We choose for $0<\delta \ll 1$ to be determined $\e=\frac \delta {\lvert y-\tilde b\rvert}$. If $\delta$ is sufficiently small, then \eqref{eq:erestr1} and \eqref{eq:erestr2} are satisfied and $\e \ll 1$.
Choosing $r\ll \delta\simeq\lvert v_\e\rvert$, but $r\sim \delta$, using that $|x'|\le 1$, \eqref{eq:main2} simplifies to
\begin{align}\label{eq:conclude1}
(\lvert \tilde b\rvert+\lvert y-\tilde b\rvert)^{p-2}\lvert y-\tilde b\rvert\lesssim&(E^\frac 1 {m}+\lvert \tilde b\rvert)^{p-2}E^\frac 1 m.
\end{align}
Using \eqref{eq:monoton} with $\beta=1$, we deduce \eqref{eq:conclude2}.

\textbf{Step 5: Proof of \eqref{eq:Linfty2} for $E\leq 1$:}
Note that in order to show \eqref{eq:conclude2} we did not use that $E\geq 1$. Thus, also if $E\leq 1$, \eqref{eq:conclude2} holds and we may assume that $\lvert y-\tilde b\rvert \lesssim 1$. In particular, it suffices to ensure $\e\ll 1$ in order for \eqref{eq:erestr1} and \eqref{eq:erestr2} to hold. 
We begin by using \eqref{eq:Young} to estimate the right-hand side of \eqref{eq:main2} (up to constant) by
\begin{multline}\label{eq:inter1}
\e^2(\lvert y-\tilde b\rvert+\lvert \tilde b\rvert)^{p-2}\lvert y-\tilde b\rvert^2+C_\e(r+\lvert \tilde b\rvert)^{p-2} r^2
+(\lvert x^\prime\rvert+\lvert \tilde b\rvert)^{p-2}\lvert x^\prime\rvert^2
+((r^{-d} E)^\frac 1 m+\lvert \tilde b\rvert)^{p-2}(r^{-d}E)^\frac 2 m\\
\lesssim \e^2(\lvert y-\tilde b\rvert+\lvert \tilde b\rvert)^{p-2}\lvert y-\tilde b\rvert^2+C_\e(r+\lvert \tilde b\rvert)^{p-2} r^2+((r^{-d}E)^\frac 1 m+\lvert \tilde b\rvert)^{p-2}(r^{-d}E)^\frac 2 m.
\end{multline}
In order to obtain the second line, we used \eqref{simeqvepsyb} together with the monotonicity from Lemma \ref{lem:monotone}. Note that for $\e \ll 1$, the first term may be absorbed on the left-hand side of \eqref{eq:main2}. Choosing $r=E^\frac 1 {m+d}$ we have shown
\begin{align*}
&(\lvert y-\tilde b\rvert+\lvert \tilde b\rvert)^{p-2}\lvert y-\tilde b\rvert^2\lesssim (E^\frac 1 {m+d}+\lvert \tilde b\rvert)^{p-2}E^\frac 2 {m+d}.
\end{align*}
Now \eqref{eq:monoton} implies the desired inequality and concludes the proof.
\end{proof}

\section{Main argument}\label{sec:mainTheorem}
This section is devoted to proving Theorem \ref{thm:main}. 

\begin{proof}[Proof of Theorem \ref{thm:main}]
Let $T$ be a minimizer for the $c$-optimal transport problem between  $\rho_0$ and $\rho_1$. Let then $u(x)=T(x)-x$, $v(y)=y-T^{-1}(y)$. Recall that we have  set 
\[
 X=\{x \ : \ |u(x)|>0\} \qquad Y=\{ y \ : \ |v(y)|>0\}.
\]
We need to show that there exists $X'$, and $Y'$ open sets such that $|X'\backslash X|=0=|Y'\backslash Y|$ and such that $T$ is a $C^{1,\alpha}$ diffeomorphism between $X'$ and $Y'$. Let $\eta>0$ and set 
\begin{multline*}
 X^\eta=\{x \ : \ |u(x)|>\eta \textrm{ and } x \textrm{ is a  Lebesgue point of } u\} \\
 Y^\eta=\{ y \ : \ |v(y)|>\eta \textrm{ and } y \textrm{ is a  Lebesgue point of } v\}.
\end{multline*}
Since up to a set of Lebesgue measure $0$, $X=\cup_{\eta>0} X^\eta$ and $Y=\cup_{\eta>0} Y^\eta$ it is enough to prove the statement with $X^\eta$ and $Y^\eta$ instead of $X$ and $Y$. Let $(\phi,\psi)$ be optimal Kantorovich potentials so that using the same convention as in \cite{DePhilippis2015,Prodhomme2021} 
\begin{equation}\label{KAntopsiphi}
 \psi(x)=\sup_y -\phi(y)-c(y-x).
\end{equation}
We claim that $\psi$ is twice differentiable a.e. in $X^\eta$ and that $\phi$ is twice differentiable a.e. in $Y^\eta$.  Notice that for $p\ge 2$,   hypothesis \eqref{hypHessc} implies that $\phi$ and $\psi$ are semi-convex so that the claim follows by Alexandrov Theorem. We may thus focus on $p\le 2$ (although the argument works also for $p\ge 2$). Since the argument for $\phi$ and $\psi$ are the same we only prove the claim for $\psi$.
We claim that for every $\bar x\in X^\eta$, there exists $R(\bar x,\eta)>0$ such that 
\begin{equation}\label{claimlargeu}
 |u(x)| \ge \eta/2 \qquad \textrm{for all } x\in B_{R(\bar x,\eta)}(\bar x).
\end{equation}
By translation we assume that $\bar x=0$ and set $b=u(0)$. Notice that by definition of $X^\eta$ we have $|b|>\eta$. Let 
\[
 E=\frac{1}{R^{p+d}}\int_{B_{2R}}|u(x)-b|^p.
\]
Note that we may assume without loss of generality that $\rho_0(0)>0$ and hence if $R$ is sufficiently small, $\inf_{B_r\Subset B_{R}}\frac{\rho_0(B_r)}{r^d}>0$ as $\rho_0\in C^0$. By Theorem \ref{thm:linfty} and scaling we then have the $L^\infty$ bound,
\[
 \sup_{B_{R/4}}|u(x)-b|\les  R \lt(E^{\frac{1}{p+d}}+E^{\frac{1}{p}}\rt).
\]
Thus if $E\le 1$ (so that $E^{\frac{1}{p}}\le E^{\frac{1}{p+d}}$) and $R$ is small enough,
\[
 \sup_{B_{R/4}}|u(x)-b|\les  R\le \eta/2
\]
whereas if $E\ge 1$ (so that $ E^{\frac{1}{p+d}}\le E^{\frac{1}{p}}$) and $R$ is small enough,
\[
 \sup_{B_{R/4}}|u(x)-b|\les  R E^{\frac{1}{p}} =\lt(\frac{1}{R^d}\int_{B_R}|u-b|^p\rt)^{1/p}\le \eta/2
\]
by definition of Lebesgue points. Using triangle inequality this concludes the proof of \eqref{claimlargeu}.\\
Let 
\[
 \tilde{X}^\eta=\cup_{\bar x\in X^\eta} B_{R(\bar x,\eta)}(\bar x)
\]
so that $\tilde{X}^\eta$ is an open set with $X^\eta\subset \tilde{X}^\eta$. Recalling that $(\phi,\psi)$ are optimal Kantorovich potentials we 
notice that for every $x$ the supremum in \eqref{KAntopsiphi} is attained at $y=T(x)$. In particular if $x\in \tilde{X}^\eta$ we may restrict the supremum to $y$ such that $|y-x|\ge \eta/2$. We finally claim that in $\tilde{X}^\eta$, $\psi$ is $C(\eta)-$semi-concave. Indeed, let $\hat{x}\in \tilde{X}^\eta$. By definition of $\tilde{X}^\eta$, there exists $\bar x\in X^\eta$ such that $\hat{x}\in B_{R(\bar x,\eta)}(\bar x)$. Let  $r\ll \eta$ be such that $B_r(\hat{x})\subset B_{R(\bar x,\eta)}(\bar x)$. Then for every $x\in B_r(\hat{x})$, if $y$ is such that $|y-\hat{x}|\le \eta/4$, we have  $|y-x|\le r+\eta/4< \eta/2$ and thus since $x\in B_{R(\bar x,\eta)}(\bar x)$, $|u(x)|\ge \eta/2 >|y-x|$. Therefore,
\[
 \psi(x)=\sup_{y\in B_{\eta/4}(\hat{x})^c} -\phi(y)-c(y-x).
\]
Now for every $(x,y)\in B_r(\hat{x})\times B_{\eta/4}(\hat{x})^c$, the function $f_y(x)=-\phi(y)-c(y-x)$ satisfies by \eqref{hypHessc}, 
\[
 |D^2 f_y(x)|\les |x-y|^{p-2}\les \eta^{p-2}
\]
and is thus semi-convex with a semi-convexity constant of the order of $\eta^{p-2}$. This concludes the proof of the claim.\\

By Alexandrov Theorem, $\psi$ is twice differentiable a.e. in $\tilde{X}^\eta$ and thus also in $X^\eta$ as claimed.\\

We may now closely follow the argument from \cite{DePhilippis2015,Prodhomme2021}. Let $X_1^\eta\subset X^\eta$ be the set of points $\bar x$ such that $\psi$ is twice differentiable at $\bar x$ and define similarly $Y^\eta_1$. We let then 
\[
 X^{'}_\eta=X_1^\eta\cap T^{-1}(Y_1^\eta) \qquad \textrm{and } \qquad Y'_\eta=Y_1^\eta\cap T(X_1^\eta).
\]
Let $(\bar x,\bar y)\in X'_\eta\times Y'_\eta$ be such that $T(\bar x)=\bar y=:b$. We claim that for $r\ll1$, $T\in C^{1,\alpha}(B_r(\bar x))$. By translation invariance of $c$ we may assume that $\bar x=0$. Since by hypothesis $0\in X'_\eta\subset X^\eta$ we have $|b|>\eta$. Writing $y=b+y'$ (and thus $T(x)=b+T'(x)$), $c_b(x,y')=c(y'+b-x)$, $\rho_{1,b}(y')=\rho_{1}(y)$ and $\psi_b(x)=\psi(x)$ we see that $\psi_b$ is a $c_b-$convex function and $T'$ is the optimal transport map between $\rho_0$ and $\rho_{1,b}$ with $T'(0)=0$. We then set 
\[
 \bar{\psi}(x)=\psi_b(x)-\psi_b(0)+c_b(x,0)-c_b(0,0)=\psi(x)-\psi(0)+c(b-x)-c(b)
 \]
 and 
 \begin{multline*}
  \bar{c}(x,y)=c_b(x,y)-c_b(x,0)-c_b(0,y)+c_b(0,0)\\
 =c(y+b-x)-c(b-x)-c(y+b)+c(b).
 \end{multline*}
 Since $\bar \psi$ is a $\bar c-$convex function we see that $T'$ is the $\bar c$-optimal transport map from $\rho_0$ to $\rho_{1,b}$. Moreover, since $\psi$ and $c(b-\cdot)$ are twice differentiable at $0$,  $\bar{\psi}$ is also twice differentiable at $0$ so that for some symmetric matrix $A$, 
 \[
  \nabla \bar \psi(x)=\nabla \bar \psi(0)+ Ax + o(|x|).
 \]
Since 
\[
 \nabla \bar \psi(0)=-\nabla_x \bar c(0,0)=0,
\]
this reduces further to 
\[
 \nabla \bar \psi(x)= Ax + o(|x|).
\]
Letting $M=-\nabla_{xy} \bar c(0,0)=\nabla^2 c(b)$ (which is non-degenerate by \eqref{hypHessc} and $|b|>\eta$) we have 
\begin{equation}\label{TaylorT'}
 T'(x)=M^{-1}Ax + o(|x|).
\end{equation}
In particular, if $r$ is small enough then for $x\in B_r$, $|T'(x)|\ll \eta\le |b|$. Arguing similarly we have that for $y\in B_r$ also $|(T')^{-1}(y)|\ll \eta\le |b|$. Let $r_0$ be a small radius such that these two properties hold and let 
\[\bar \rho_0=\rho_0 \chi_{B_{r_0}} \qquad \textrm{and } \qquad  \bar \rho_{1,b}=T'\# \bar \rho_0.\]
By the previous considerations, if $r\ll r_0$ then 
\begin{equation}\label{barequalnobar}
 \bar \rho_{1,b}=\rho_{1,b} \qquad \textrm{in } B_r.
\end{equation}
Moreover, by \eqref{hypHessc},
\begin{equation}\label{mixderc}
 -\nabla_{xy} \bar c(x,y)=\nabla^{2} c(y+b-x)\sim_b {\rm Id} \qquad \textrm{in } B_{r_0}\times B_{r_0}.
\end{equation}
We now make the change of variables (arguing as in \cite[Proof of Corollary 1.4]{Prodhomme2021} we have that $A$ is positive definite) $\tilde{x}=A^{\frac{1}{2}}x$ and $\tilde{y}= A^{-\frac{1}{2}}My$ and thus
\[
 \tilde{T}(\tilde x)= A^{-\frac{1}{2}} M T'(A^{-\frac{1}{2}}\tilde{x}), \qquad \tilde{c}(\tilde{x},\tilde{y})= \bar c(A^{-\frac{1}{2}} \tilde x, M^{-1} A^{\frac{1}{2}}\tilde{y}).
\]
Setting 
\[
 \tilde{\rho}_0(\tilde{x})=\frac{\rho_0(A^{-\frac{1}{2}} \tilde{x})}{\rho_0(0)} \qquad \textrm{and } \qquad \tilde{\rho}_1(\tilde{y})=\frac{\bar \rho_{1,b}(M^{-1} A^{\frac{1}{2}} \tilde{y})}{\bar \rho_{1,b}(0)}
\]
we may check as in \cite[Proof of Corollary 1.4]{Prodhomme2021} that $\tilde{T}$ is the optimal transport map between $\tilde{\rho}_0$ and $\tilde{\rho}_1$ for the cost $\tilde{c}$. By \eqref{barequalnobar} we have for $r\ll r_0$ that $\tilde{\rho}_i\in C^\alpha(B_r)$. Moreover, by definition, $\tilde{\rho}_0(0)=\tilde{\rho}_1(0)=1$. With this change of variables \eqref{TaylorT'} becomes 
\[
 \tilde{T}(\tilde{x})=\tilde{x}+o(|\tilde{x}|)
\]
so that 
\[
\lim_{R\to 0} \frac{1}{R^{2+d}}\int_{B_R} |\tilde{T}-\tilde{x}|^2 \tilde{\rho}_0=0.
\]
Finally, by  \eqref{mixderc}, if $r\ll r_0$,
\begin{itemize}
 \item $\tilde{c}\in C^{2,\alpha}(B_r\times B_r)$;
 \item  for $x\in B_r$, the map $y\mapsto \nabla_x \tilde{c}(x,y)$ is one-to-one from $B_r$ into $\R^d$;
 \item for $y\in B_r$, the map $x\mapsto \nabla_y \tilde{c}(x,y)$ is one-to-one from $B_r$ into $\R^d$;
 \item $\det \nabla_{xy} \tilde{c} \neq 0$ for $(x,y)\in B_r\times B_r$.
\end{itemize}
We may thus apply 
 \cite[Theorem 1.1]{Prodhomme2021} and conclude that $T'\in C^{1,\alpha}(B_{R})$ for $R\ll r$ small enough. Returning back to the original variables we find that $T\in C^{1,\alpha}(B_R)$ for a (non relabelled) $R$ small enough. Moreover, if $R$ is sufficiently small, $B_R\subset X'_\eta$. Arguing similarly for $T^{-1}$ we conclude that $T$ is a diffeomorphism between $B_{R}$ and $T(B_R)$. In particular $B_R\times T(B_R) \subset X'_\eta\times Y'_\eta$ which is open and it follows that $T$ is a global $C^{1,\alpha}$ diffeomorphism between $X'_\eta$ and $Y'_\eta$.
 \end{proof}
\section*{Acknowledgments}
We warmly thank F. Otto for many useful discussions on this topic.


\begin{thebibliography}{}
\bibitem{Mingione2002} Acerbi,~E. and Mingione,~G.,
\newblock Regularity Results for Stationary Electro-Rheological Fluids
\newblock \emph{Arch. Rational Mech. Anal.} 164, (2002).

\bibitem{Bouchitte2007} Bouchitté,~G., Jimenez,~C. and Rajesh,~M.
\newblock A new $L^\infty$ estimate in optimal mass transport
\newblock \emph{Proc. Amer. Math. Soc.} 135(11), 3525--3535 (2007).


\bibitem{Cafp} Caffarelli,~L., González,~M.~D.~M. and  Nguyen,~T. 
\newblock A perturbation argument for a Monge–Amp\`ere type equation arising in optimal transportation.
\newblock \emph{Archive for rational mechanics and analysis}, 212, 359-414 (2014).

\bibitem{DePhilippis2015} De Philippis,~G. and Figalli,~A.
\newblock Partial regularity for optimal transport maps
\newblock \emph{Publ. Math. Inst. Hautes Études Sci.} 121, 81--112 (2015).

\bibitem{DuzMin} Duzaar,~F. and Mingione,~G. 
\newblock Regularity for degenerate elliptic problems via p-harmonic approximation.
\newblock \emph{Annales de l'Institut Henri Poincaré C, Analyse non linéaire}, Vol. 21, No. 5, 735-766 (2004). 

\bibitem{Figalli2010} Figalli,~A. and  Kim,~Y.-H.
\newblock Partial regularity of Brenier solutions of the Monge-Ampère equation
\newblock \emph{Discrete Contin. Dyn. Syst.} 28(2), 559--565 (2010).

\bibitem{Goldmanrendiconti} Goldman,~M. 
\newblock An $\e$-regularity result for optimal transport maps between continuous densities.
\newblock \emph{Rendiconti Lincei. Matematica e Applicazioni}, (2020).

\bibitem{GHO} Goldman,~M., Huesmann,~M. and Otto,~F. 
\newblock Quantitative Linearization Results for the Monge‐Amp\`ere Equation.
\newblock \emph{Communications on Pure and Applied Mathematics}, 74(12), 2483-2560 (2021).

\bibitem{Goldman2018} Goldman,~M. and Otto,~F.
\newblock A variational proof of partial regularity for optimal transportation maps.
\newblock \emph{Ann. Sci. Ec. Norm. Super.} 53(5), 1209--1233(2020).

\bibitem{Gutierrez2021} Guti{\'e}rrez,~C.E. and Montanari, A.
\newblock $L^\infty$-estimates in optimal transport for non quadratic costs.
\newblock \emph{Calc. Var.} 61,163 (2022).

\bibitem{Gutierrez2022} Guti{\'e}rrez,~C.E. and Montanari, A.
\newblock Fine properties of monotone maps arising in optimal transport for non-quadratic costs.
\newblock \emph{arXiv} 2208.00193, (2022).

\bibitem{Gut2024} Guti{\'e}rrez,~C.E. and Montanari, A.
\newblock Differentiability of monotone maps related to  non-quadratic costs.
\newblock In preparation.

\bibitem{Iwaniec1989} Iwaniec,~T. and Manfredi,~J.
\newblock Regularity of p-harmonic functions on the plane
\newblock \emph{Rev. Math. Iberoam.} 5(1-2), 1-10 (1989).

\bibitem{Koch2024} Koch,~L. 
\newblock Geometric linearisation for optimal transport with strongly p-convex cost.
\newblock \emph{Calc. Var.} 63, 87 (2024).


\bibitem{Prodhomme2021}
F.~Otto, M.~Prod'homme, and T.~Ried
\newblock{Variational Approach to Regularity of Optimal Transport Maps: General Cost Functions}.
\newblock {\em Ann. PDE} 7(17), 2021.

\bibitem{Villani} Villani,~C.
\newblock Topics in optimal transportation
\newblock \emph{ Amer. Math. Soc., Providence, RI} Graduate Studies in Mathematics 58(2003).


\end{thebibliography}
\end{document}